\documentclass[12pt,a4paper]{amsart}
\usepackage{graphics}
\usepackage{epsfig}
\usepackage{graphicx}
\theoremstyle{plain}
\usepackage{amssymb}
\usepackage[ukrainian,english]{babel}
\usepackage[usenames]{color}
\usepackage{colortbl}
\advance\hoffset-20mm \advance\textwidth40mm

\newtheorem{theorem}{Theorem}
\newtheorem{lemma}{Lemma}
\newtheorem*{theo*}{Theorem}

\theoremstyle{definition}

\newtheorem*{definition*}{Definition}

\DeclareMathOperator{\ad}{ad}

\newcommand{\p}{\partial}

\begin{document}
\sloppy
\title[On maximality of some solvable and locally nilpotent ...]
{On maximality of some solvable and locally \\ nilpotent   subalgebras of the Lie algebra $W_n(K)$}
\author
{D.Efimov, M.Sydorov, K.Sysak}
\address{D.Efimov: Department of Algebra and Computer Mathematics, Faculty of Mechanics and Mathematics,
Taras Shevchenko National University of Kyiv, 64, Volodymyrska street, 01033  Kyiv, Ukraine
}
\email{danil.efimov@yahoo.com}
\address{M.Sydorov:
Department of Algebra and Computer Mathematics, Faculty of Mechanics and Mathematics,
Taras Shevchenko National University of Kyiv, 64, Volodymyrska street, 01033  Kyiv, Ukraine}
\email{smsidorov95@gmail.com}
\address{K.Sysak:
 Department of Higher and Applied Mathematics, Education and Research Institute of Energetics, Automatics and Energy saving, National University of Life and Environmental Sciences of Ukraine, 15, Heroiv Oborony street, 03041  Kyiv, Ukraine}
\email{sysakkya@gmail.com}
\date{\today}

%
\begin{abstract}

{Let $K$ be an algebraically closed field of characteristic zero,  $P_n=K[x_1,\ldots ,x_n]$  the polynomial ring, and  $W_n(K)$  the Lie algebra of all $K$-derivations on $P_n$.   One of the most important subalgebras of $W_n(K)$ is the triangular subalgebra $u_n(K) = P_0\partial_1+\cdots+P_{n-1}\partial_n$, where $\partial_i:=\partial/\partial x_i$ are partial derivatives on $P_n$. This subalgebra consists of locally nilpotent derivations on $P_n.$ Such derivations  define automorphisms of the ring $P_n$ and were studied by many authors. The  subalgebra $u_n(K) $ is contained in another interesting subalgebra $s_n(K)=(P_0+x_1P_0)\partial_1+\cdots +(P_{n-1}+x_nP_{n-1})\partial_n,$ which  is solvable of the derived length $ 2n$ that is the maximum derived length of solvable subalgebras of $W_n(K).$
It is proved that $u_n(K)$  is a maximal locally nilpotent subalgebra and $s_n(K)$ is a maximal solvable subalgebra of the Lie algebra $W_n(K)$.}
 \end{abstract}
\maketitle

\section{Introduction}

Let $K$ be an algebraically closed field of characteristic zero and $P_n=K[x_1,\ldots ,x_n]$  the polynomial ring in $n$ variables. Recall that a $K$-linear map $D: P_n \to  P_n$ is called a $K$-derivation (or simply  a derivation if the field $K$ is fixed) if it satisfies the Leibniz rule: $D(fg)=D(f)g+fD(g)$ for any $f, g \in P_n$. For any $f_1,\ldots ,f_n \in P_n$ there exists a unique $K$-derivation ${D} \in W_n(K)$ of the form ${D} = f_1\partial_1+\cdots +f_n\partial_n$ such that ${D}(x_i)=f_i, i=1,\ldots ,n,$ where $\partial_i:=\partial/\partial x_i$ are partial derivatives on $P_n$. The vector space $W_n(K)$ of all $K$-derivations on $P_n$ is a Lie algebra over the field $K$ with respect to the Lie bracket $[{D}_1, {D}_2]={D}_1{D}_2-{D}_2{D}_1$, ${D}_1, {D}_2 \in W_n(K)$. This Lie algebra is of great interest in many areas of mathematics and physics because in geometric language any derivation can be considered as a vector field on $K^n$ with polynomial coefficients.

Nilpotent, locally nilpotent and solvable subalgebras of $W_n(K)$ were studied by many authors, started  from~\cite{Lie1} (see, for example,~\cite{Bavula, Olver2, MP1}).
One of the most important subalgebras of $W_n(K)$ is the triangular Lie algebra $$u_n(K)=P_0\partial_1+\cdots +P_{n-1}\partial_n,$$ which consists of locally nilpotent derivations on $K[x_1,\ldots ,x_n].$ This Lie algebra is locally nilpotent but not nilpotent, its structure and properties were studied in~\cite{Bavula}. We consider its embedding in $W_n(K)$ and prove that $u_n(K)$ is a maximal locally nilpotent subalgebra of $W_n(K)$ (Theorem~1). Another maximality property of $u_n(K)$ was considered in~\cite{Skutin}, where it was proved that $u_n(K)$ is a maximal subalgebra contained in the set of locally nilpotent derivations on $P_n$ (note that this set is not a Lie subalgebra of $W_n(K)$).

In~\cite{MP1}, it was proved that the derived length of solvable subalgebras in $W_n(K)$ does not exceed $2n$. The known example of solvable subalgebras that reaches this bound was pointed out in~\cite{Martello}, this is the subalgebra
$$s_n(K)=(P_0+x_1P_0)\partial_1+\dots+(P_{n-1}+x_nP_{n-1})\partial_n.$$
It is clear that the subalgebra $u_n(K)$ is properly contained in $s_n(K)$.
The subalgebra $s_n(K)$  has also a maximality property: we prove that $s_n(K)$ is a maximal solvable subalgebra of $W_n(K)$ (Theorem 2). Note that $s_n(K)$ appears in a natural way while studying Lie algebras of vector fields on $\mathbb{C}^n$ (see~\cite{Martello}). In general, maximal subalgebras of the Lie algebra $W_n(K)$ are not described, but some types of such subalgebras are known (see, for example,~\cite{Amemiya}). Note that  the structure of maximal subalgebras of semisimple Lie algebras was described in~\cite{Dynkin}.

We use standard notations. Recall that a derivation  $D \in W_n(K)$ is called locally nilpotent if for any $f \in P_n$ there exists a positive integer $k=k(f)$ such that $D^k(f)=0$. Let a derivation $D \in W_n(K)$ be written in the form $D = f_1\partial_1+\dots +f_n\partial_n$, $f_i \in P_n,$ $i=1,\dots, n.$ Then we say that $D$ has an index $k$ if $f_k \neq 0$ and $f_m= 0$ for all $m>k$. Let  $f=f(x_1,\ldots ,x_n)\in P_n$ be a polynomial. Then we say that $f$ has an index $s$ if $\frac{\partial f}{\partial x_s}\neq 0$ but $\frac{\partial f}{\partial x_i}=0$ for $i>s$. If~$D_1, \ D_2 \in W_n(K)$ then we write $[{D}_1^k, {D}_2]=\underbrace{[{D}_1,[ \ldots [{D}_1 }_{k},{D}_2]\ldots ]\ldots  ]$. We denote as usual by  $K^*$ the multiplicative group of the field $K$, that is $K^*=K\setminus{\{0\}}.$

\section{Maximality of $u_n(K)$}

We need some technical lemmas to prove Theorem~1, the main result of this section.
 Lemma~\ref{diff} seems to be known but having no exact references  we point out its proof for completeness.

\begin{lemma}[\cite{MP1}, Lemma 1]\label{commuting}
Let $D_{1}, D_{2}\in W_n(K)$  and $a, b\in K[x_1, \ldots , x_n].$ Then it holds:

(1)  $[aD_{1}, bD_{2}]=ab[ D_1, D_2]+aD_1(b)D_2-bD_2(a)D_{1}.$

(2) If $[D_{1}, D_{2}]=0,$  then  $[aD_{1}, bD_{2}]=aD_1(b)D_2-bD_2(a)D_{1}.$

\end{lemma}

\begin{lemma}\label{diff}

Let $f \in P_n = K[x_1, \ldots, x_n]$, $\deg f\geq 1$. Then:

(1) there exist nonnegative integers $\alpha_1,\ldots ,\alpha_n$ such that $\partial_1^{\alpha_1}\dots \partial_n^{\alpha_n}(f)$ is a nonzero constant;

(2) if $\deg _{x_i}f \geq 1$, then there exist nonnegative integers $\beta_1,\ldots,\beta_n$ (depending on $i$) such that $\partial_1^{\beta_1}\dots\partial_n^{\beta_n}(f)=\lambda_i x_i+g_i(x_1, \ldots , x_{i-1}, x_{i+1}, \ldots ,x_n), $ where $ \lambda_i \in K^*$.

\end{lemma}

\begin{proof}

(1)  Let $ d=\deg f$ and $f=f_0+\cdots +f_d$ be the sum of homogeneous components of $f$. Choose any monomial $ax_1^{\alpha _1}\dots x_n^{\alpha _n}$ of the homogeneous polynomial $f_d$. Then $\partial_1^{\alpha_1}\dots \partial_i^{\alpha _n}(a x_1^{\alpha_1}\dots x_n^{\alpha_n})=\gamma$ for some $\gamma \in K^*$. If there exists another  monomial $bx_1^{\beta_1}\dots x_n^{\beta_n}$ of the polynomial $f_d$ then the polynomial $\partial_1^{\alpha_1}\dots \partial_i^{\alpha _n}(bx_1^{\beta_1}\dots x_n^{\beta_n})$ is a constant. This constant is nonzero only in the  case  when $\beta _1=\alpha _1, \ldots , \beta _n=\alpha _n.$ The latter is impossible because of the choice of the monomial $bx_1^{\beta_1}\cdots x_n^{\beta_n}$. So $\partial_1^{\alpha_1}\dots \partial_i^{\alpha _n}(f_d)=\gamma,$ and since the equalities $\partial_1^{\alpha_1}\dots \partial_i^{\alpha _n}(f_i)=0$ hold for all $i<d,$ we get $\partial_1^{\alpha_1}\dots \partial_i^{\alpha _n}(f)=\gamma, \ \gamma \in K^*.$

(2) Let $\deg _{x_{i}} (f)=d\geq 1.$ Expand $f$ in powers of $x_i:$ $f=h_0+h_1x_i+\cdots + h_dx_i^d$, where $  \deg _{x_{i}} (h_j)=0, j=1, \ldots ,d.$  Then
$$\partial _i^{d-1}(f)=t_0(x_1, \ldots x_{i-1}, x_{i+1}, \ldots , x_n)+t_1(x_1, \ldots x_{i-1}, x_{i+1}, \ldots , x_n)x_i.$$
 If the polynomial $t_1( x_1, \ldots x_{i-1}, x_{i+1}, \ldots , x_n)$ is nonconstant then by the first part of this lemma there exist nonnegative integers $\beta _1, \ldots \beta _{i-1}, \beta _{i+1}, \ldots ,\beta _n$ such that $$\partial_1^{\beta_1}\dots \partial _{i-1}^{\beta _{i-1}} \partial_{i+1}^{\beta _{i+1}}\dots\partial_{n}^{\beta _{n}}(t_1)=\lambda _i, \ \lambda _i \in K^*.$$ Denoting  $\beta _i=d-1$ we get $$\partial_1^{\beta_1}\dots \partial _{n}^{\beta _{n}}(f)=\lambda _ix_i+g_i(x_1, \ldots x_{i-1},x_{i+1}, \ldots ,x_n),$$ where  $g_i= \partial_1^{\beta_1}\dots \partial _{i-1}^{\beta _{i-1}} \partial_{i+1}^{\beta _{i+1}}\dots \partial_{n}^{\beta _{n}}(t_0), $  $\lambda _i\in K^*.$

\end{proof}

\begin{lemma}\label{lem2}

If there exists a locally nilpotent subalgebra $S$ of the Lie algebra $W_n(K)$ that properly contains $u_n(K)$, then there exists a (nonzero) linear derivation $D \in S \setminus u_n(K)$ of the form
$D = \sum_{i, j = 1}^n \lambda_{ij}x_j\partial_i$, where $\lambda_{ij} = 0$ for all $i>j$.
\end{lemma}

\begin{proof}
Suppose the statement of the lemma is false and the set $ S \setminus u_n({K})$ does not contain  any nonzero  linear derivation. Let us choose a derivation $D\in S \setminus u_n(K)$ of minimum degree and write it in the form $D=f_1\partial _1+\cdots +f_n\partial _n$, $f_i\in K[x_1, \ldots , x_n].$ Since $\partial _i \in S, i=1,\ldots ,n,$ we have  that
$$[\partial _i, D]=\partial _i(f_1)\partial _1+\cdots +\partial _i(f_n)\partial _n\in S $$ and $ \deg [\partial _i, D]<\deg D.$ By the choice of $D,$ we see that $[\partial _i, D]\in u_n(K), i=1, \ldots ,n.$ Let us show that
$\deg _{x_j}f_i\leq 1$ for $j\geq i.$
Indeed, if $\deg _{x_j}f_i\geq 2$, then  $\deg _{x_j}\partial _j(f_i)\geq 1$, which contradicts  the inclusion $[\partial _j, D]\in u_n(K)$ mentioned above.

 Write the polynomial $f_i$ in powers of $x_j\colon$  $ f_i=f_{i0}+f_{i1}x_j$ for some $f_{i0}, f_{i1}\in K[x_1, \ldots ,x_{j-1}, x_{j+1}, \ldots , x_n]$. If the polynomial $f_{i1}$ is nonconstant, then $\partial _k(f_{i1})\not =0$ for some $k, 1\leq k\leq n.$ But then $ \deg _{x_j}\partial _k(f_i)=1$ and therefore $[\partial _k, D]\not\in u_n(K).$
The latter contradicts the proven above. So, $f_i=f_{i0}+\lambda _{ij}x_j$ for some $\lambda _{ij}\in K.$ Repeating these considerations for every $j\geq i$ we see that $f_i$ can be chosen in the form $f_i=\sum_{i, j = 1}^n \lambda_{ij}x_j+\overline{f_{i}},$  where $\lambda_{ij} = 0$ for all $i>j$  and $\overline{f_{i}}$ does not depend on $x_j$, $j\geq i,$ i.e.,  $\overline{f_{i}} \in K[x_1, \ldots , x_{i-1}].$
But then $\overline{f_{i}}\partial _i\in u_n(K)$ and we can subtract $\overline{f_{i}}\partial _i$ from $D.$ Applying such considerations to all $f_i,$  we get the statement of the lemma.

\end{proof}

\begin{theorem}\label{th1}
The triangular subalgebra $u_n(K)$ is a maximal locally nilpotent subalgebra of  the Lie algebra  $W_n(K)$.
\end{theorem}

\begin{proof}
Suppose to the contrary that $u_n(K)$ is properly contained in a locally nilpotent subalgebra $S$ of $W_n(K)$. By Lemma~\ref{lem2}, there exists an element $D \in S \setminus u_n(K)$ of the form
$$D = \sum_{i=1}^n \sum_{j=i}^n \lambda_{ij} x_j \partial_i,  \ \lambda_{ij} \in K,$$ i.e. $D = f_1 \partial_1 + \dots + f_n \partial_n $, where
$f_i = \lambda_{ii}x_i + \lambda_{i, i+1}x_{i+1} + \dots + \lambda_{in}x_n$.

Firstly, let us prove that every linear derivation $D\in S \setminus u_n(K)$ is diagonal, i.e. that the matrix $(\lambda_{ij})_{i,j=1}^n$ is diagonal. Let it be not the case and choose  any linear (non-diagonal) derivation $D = \sum_{i=1}^n \sum_{j=i}^n \lambda_{ij} x_j \partial_i, \ \lambda_{ij} \in K$ from the set $ S \setminus u_n(K)$. Note that the derivation $x_i \partial_j \in u_n(K)$ for $i<j$ and consider the product

\begin{multline}\label{eq2} [x_i\partial _j, D]=[x_i\partial_j, \sum_{s=1}^n \lambda _{1s}x_s \partial_1+\cdots +\sum_{s=k}^n \lambda _{ks}x_s \partial_k+ \cdots +\lambda _{nn}x_n\partial _n] = \\
= x_i \lambda_{1j} \partial_1 +\cdots + x_i\lambda_{ij}\partial _i + \dots + x_i\lambda_{jj}\partial_j
- ( \lambda _{1i}x_i+ \dots + \lambda_{ji}x_i)\partial_j.
\end{multline}

Denote this product $[x_i\partial _j, D]$ by $D_0$. Using (\ref{eq2}) one can  easily  see that $[D_0, x_i\partial _j]=\lambda_{ij}x_i\partial _j, $ i.e.
$x_i\partial _j$ is an eigenvector for the linear operator $\ad D_0$ with the eigenvalue $\lambda_{ij}.$ Since $D$ is non-diagonal (by our assumption) there exists a nonzero coefficient $\lambda _{ij}, i<j.$ The latter is impossible because the subalgebra $S$ is locally nilpotent. The obtained contradiction shows that all linear derivations from $S \setminus u_n(K)$ are diagonal.

Take any linear derivation $D\in S \setminus u_n(K)$,  $D = \mu_1 x_1 \partial_1 + \dots + \mu_n x_n \partial_n.$ Obviously, $  D \neq 0$. Let us show that $S \setminus u_n(K)$ contains a derivation $D_1 = \mu E_n$, where $E_n = x_1 \partial_1 + \dots + x_n \partial_n$ is the Euler derivation. If $D = \sum \mu_i x_i \partial_i$ is not proportional to $E_n$, then there exist $\mu_i, \ \mu_j,$ such that $\mu_i \neq \mu_j, \ i < j$. Then $S$ contains the product
$$[\sum_{i=1}^n \mu_i x_i \partial_i, x_i\partial_j] = (\mu_i - \mu_j)x_i \partial_j, \ \mu_i - \mu_j \neq 0.$$

The latter means that $x_i \partial_j$ is an eigenvector for the linear operator $\ad D$ with the (nonzero) eigenvalue $\mu_i - \mu_j$, which is impossible because $S$ is a locally nilpotent subalgebra of $W_n({K})$. Therefore, we have  that $D = \mu E_n$ for some $\mu \in K^*$. But $[\mu E_n, x_1^2\partial_2] = \mu x_1^2 \partial_2$ for the element $x^2_1 \partial_2 \in u_n({K}) \subset S$. The latter is impossible as it was mentioned above. The obtained contradiction shows   that $S = u_n({K})$ and $ u_n({K})$ is a maximal locally nilpotent subalgebra of the Lie algebra  $W_n({K})$.

\end{proof}

\section{Maximality of $s_n(K)$}

Recall that we denote by $P_i=K[x_1, \dots, x_i]$ the polynomial ring over $ K.$ We also denote for convenience $P_0= K.$ It is easy to see that the $ K$-subspace
$$s_n( K)=(P_0+x_1P_0)\p_1+\dots+(P_{n-1}+x_nP_{n-1})\p_n$$
is a subalgebra of $W_n( K)$ and $u_n( K)\subset s_n(K).$ This subalgebra is solvable of the derived length~$2n.$ Some properties of $s_n(K)$ were pointed out in~\cite{Martello}.
Since the derived length of solvable subalgebras of $W_n(K)$ does not exceed~$2n$ (see~\cite{MP1},~\cite{Martello}), the subalgebra $s_n(K)$ has the maximum possible derived length.

Here we prove that $s_n(K)$ is a maximal solvable subalgebra of $W_n(K)$  (Theorem~\ref{th2}).

Let $D\in W_n( K)$ be a derivation of the form $D=f_1\p_1+\dots+f_n\p_n,$ where $f_i \in P_n,$ $i=1,\ldots ,n.$ Recall that we say that  $D$ has an index $k$ if $f_k\neq 0$ and $f_m=0$ for all $m>k.$  We also say that a polynomial $f\in P_n$  has an index $s$ if $\frac{\p f}{\p x_s}\neq 0$ and $\frac{\p f}{\p x_i}=0$ for all $i>s.$

\begin{lemma}\label{sl2}
Let $T_1=\sum _{i=1}^{k-1}g_i\partial _i+\partial _k$, $T_2=\sum _{i=1}^{k-1}h_i\partial _i-x_k^2\partial _k$, $T_3=\sum _{i=1}^{k-1}f_i\partial _i-2x_k\partial _k$ be  derivations from $W_n(K)$ for some $k\leq n$. Then $T_1, T_2, T_3$ generate a non-solvable subalgebra of the Lie algebra $W_n(K)$.
\end{lemma}
\begin{proof}
Direct calculations show that
$[T_1, T_2]=\sum _{i=1}^{k-1}a_i\partial _i-2x_k\partial _k$, $[T_3, T_1]=\sum _{i=1}^{k-1}b_i\partial _i+2\partial _k$, $[T_3, T_2]=\sum _{i=1}^{k-1}c_i\partial _i+2x_k^2\partial _k$  for some polynomials $a_i, b_i, c_i\in K[x_1, \ldots , x_n] .$ Denote by $L$ the subalgebra of the Lie algebra $W_n(K)$ generated by the elements $T_1, T_2, T_3$ and by $L_1$ the subalgebra generated by $\partial _k,  -x_k^2\partial _k, -2x_k\partial _k.$ Define a map   $\varphi $ from the set $\{ T_1, T_2, T_3\}$ onto the set $\{    \partial _k,  -x_k^2\partial _k, -2x_k\partial _k \}$ by the rule:
$$ \varphi (T_1)= \partial _k,  \ \varphi (T_2)= -x_k^2\partial _k,  \  \varphi (T_3)=-2x_k\partial _k.$$
  The rule of commutation of generators  $T_1, T_2, T_3$ shows that $\varphi$ can be  extended to a homomorphism of the Lie algebra  $L$ onto $L_1$. Since $L_1$ is isomorphic to $sl_2(K)$  we conclude that $L$ is non-solvable.
\end{proof}
\begin{theorem}\label{th2}
	The subalgebra $s_n( K)$ is a maximal solvable subalgebra of the Lie algebra  $W_n( K).$
\end{theorem}

\begin{proof}
	Suppose to the contrary that there exists a solvable subalgebra $S\subset W_n( K)$ such that $s_n(K)$ is properly contained in $S.$
	 Denote by $k$  the smallest index of derivations from the set  $S\setminus s_n( K)$ and  consider the set $\mathfrak D_k$ of all derivations $D\in S\setminus s_n( K)$ that have the index $k.$ Let us choose a derivation $D\in \mathfrak D_k$  in such a way that its (nonzero)  polynomial coefficient $f_k$ (by the partial derivative $\partial _k$) has the smallest index $s.$ Then we have
  $$D=f_1\partial _1+\dots +f_k\partial _k,$$ where $f_i\in P_n, \ i=1, \ldots,k,$
 $\frac {\p f_k}{\p x_s}\not =0$
 and $\frac {\p f_k}{\p x_i}=0$ for all $i>s.$
	
	Firstly, let us show that $s\geq k$ and if $s=k$ then $\deg_{x_s}f_k\geq 2.$ 	
	Indeed, if $s<k$ then $f_k\p_k\in u_n( K),$ since $f_k\in P_{s}$ and $s<k.$   By our assumption, $u_n( K)\subset S.$ Then $D-f_k\p_k \in S\setminus s_n( K)$ and  this derivation has an index less than $k,$ which contradicts  our choice of $D.$ Therefore, $s\geq k.$ Let  $s=k$. Then $x_s=x_k,$ $\frac{\p f_k}{\p x_k}\neq 0$ and $f_k\in P_k.$ Let us expand the polynomial $f_k$ in powers of~$x_k:$
	$$f_k(x_1,\dots, x_k)=h_0(x_1, \dots, x_{k-1})+h_1(x_1, \dots, x_{k-1})x_k+\dots +h_t(x_1, \dots, x_{k-1})x_k^t$$
	for some $t\geq 1$ and polynomials $h_i\in P_{k-1},$ $i=1\ldots, t.$
	If $t=1$ then $$f_k\p_k=(h_0(x_1, \dots, x_{k-1})+h_1(x_1, \dots, x_{k-1})x_k)\p_k\in s_n( K).$$ It is obvious that   $D-f_k\p_k\in S\setminus s_n( K)$ and the derivation $  D-f_k\p_k$ has the index less than $k.$ The latter contradicts  the choice of $D$. Thus, if  $s=k$ then  $\deg_{x_k}f_k\geq 2.$
	
	Therefore, we can write $D\in \mathfrak D_k$ in the form
	\begin{equation}\label{th3_eq1}
	D=f_1\p_1+\dots+f_k\p_k,
	\end{equation} where $\deg_{x_s}f_k\neq 0,$ $\deg_{x_i}f_k=0$ for all $i>s$ and $s\geq k.$
	
	We investigate the possible two cases: $s>k$ and $s=k.$
	
	\underline{\it Case 1.} Let us begin with the case $s>k.$
	
	Let us  expand the polynomial $f_k\in P_{s}$  from the derivation $D$ (written in the form  ~(\ref{th3_eq1})) in  powers of~$x_s:$
	$$f_k(x_1, \dots, x_s)=g_0(x_1, \dots, x_{s-1})+g_1(x_1, \dots, x_{s-1})x_s+\dots+g_l(x_1, \dots, x_{s-1})x_s^l,$$ for some $l\geq 1$,
where $g_i\in P_{s-1}, \ g_l\neq 0.$
	Then the  product $D_0=[\underbrace{\p_s, \dots, [\p_s,}_{l-1~\text{times}} D]\dots ]\in S$ can be written in the form $$D_0=\alpha_1 \p_1+\dots +\alpha_{k-1}\p_{k-1}+(u_0(x_1, \dots, x_{s-1})+u_1(x_1, \dots, x_{s-1})x_s)\p_k,$$ for some  $\alpha_i\in P_n, \ i=1,\ldots ,k-1,$ $u_0, u_1\in P_{s-1}$ and $u_1\neq 0$ by the choice of $s.$
	
	By Lemma~\ref{diff},  there exists a differential operator $\p_1^{\beta_1}\dots\p_{s-1}^{\beta_{s-1}},$  $\beta_i\geq 0$ such that $$\p_1^{\beta_1}\dots\p_{s-1}^{\beta_{s-1}}(u_1)=\lambda\in \ K^*,$$
	where $ K^*$ is the  group of units of the field $K.$ Applying this operator to the derivation $D_0$ we obtain a derivation $D_1$ of the form
 $$D_1=\gamma_1 \p_1+\dots +\gamma_{k-1}\p_{k-1}+(v_0(x_1, \dots, x_{s-1})+\lambda x_s)\p_k,$$
	for some  $\gamma_i\in P_n, \ i=1, \ldots ,k-1,$ $v_0 \in P_{s-1}$ and $\lambda \neq 0.$ The derivation $D_1\in S\setminus s_n( K),$ since $\lambda\neq 0$ and $s>k.$
	
	Consider the subcase $s-1>k.$ Then it holds
	\begin{multline}\label{th3_eq2}
	[x_{s-1}\p_s; (v_0(x_1, \dots, x_{s-1})+\lambda x_s)\p_k]=\\ =x_{s-1}\p_s(v_0(x_1, \dots, x_{s-1})+\lambda x_s)\p_k-(v_0(x_1, \dots, x_{s-1})+\lambda x_s)\p_k(x_{s-1})\p_s=\lambda x_{s-1}\p_k,
	\end{multline}
	since  $\p_k(x_{s-1})=0$ for $s-1>k.$ Moreover,  for all $i=1, \ldots ,k-1$ we get
	\begin{equation}\label{th3_eq3}
	[x_{s-1}\p_s; \gamma_i \p_i]=x_{s-1}\p_s(\gamma_i)\p_i - \gamma_i \p_i(x_{s-1})\p_s=x_{s-1}\p_s(\gamma_i)\p_i.
	\end{equation}
	Taking into account the relations~(\ref{th3_eq2}) and (\ref{th3_eq3}), we obtain $$[x_{s-1}\p_s; D_1]=\sum_{i=1}^{k-1} x_{s-1}\p_s(\gamma_i)\p_i+ \lambda x_{s-1}\p_k.$$
	Since $s-1>k$ we have that $[x_{s-1}\p_s; D_1]\not \in s_n(K).$
	Therefore, $[x_{s-1}\p_s; D_1]\in S\setminus s_n(K),$ and we get a contradiction, since the polynomial $\lambda x_{s-1}$ has the index less than $s.$
	
	Now let  $s-1=k.$ Then  $\p_k(x_{s})=0,$ so we get
	\begin{multline}\label{th3_eq4}
	[x_{s}\p_s; (v_0(x_1, \dots, x_{s-1})+\lambda x_s)\p_k]=\\ =x_{s}\p_s(v_0(x_1, \dots, x_{s-1})+\lambda x_s)\p_k-(v_0(x_1, \dots, x_{s-1})+\lambda x_s)\p_k(x_{s})\p_s=\lambda x_{s}\p_k.
	\end{multline}
	Since $\p_i(x_s)=0$ for all $i=1, \ldots , k-1$  one can easily show (using the relation~(\ref{th3_eq4})) that
	$$ [(1/{\lambda })x_{s}\p_s; D_1]=(1/{\lambda })\sum_{i=1}^{k-1} x_{s}\p_s(\gamma_i)\p_i+ x_{s}\p_k.$$
	Denote $D_2=(1/{\lambda })\sum_{i=1}^{k-1} x_{s}\p_s(\gamma_i)\p_i+ x_{s}\p_k.$  It is obvious that $D_2\in S\setminus s_n(K)$ because $s=k+1.$  Note that $x_k^2\partial _{k+1}\in s_n(K)$ and therefore
$$[x_k^2\partial _{k+1}, D_2]=[x_k^2\partial _{k+1}, (1/{\lambda })\sum _{i=1}^{k-1}x_s\partial _s(\gamma _i)\partial _i+x_{k+1}\partial _k]=\sum _{i=1}^{k-1}\alpha _i\partial _i+[x_k^2\partial _{k+1}, x_{k+1}\partial _k]$$
  for some $\alpha _i\in K[x_1, \ldots ,x_n]$ is an element of subalgebra $S.$

But $[x_k^2\partial _{
k+1},  x_{k+1}\partial _k]=x_k^{2}\partial _k-2x_kx_{k+1}\partial _{k+1}.$  Therefore
$$[x_k^2\partial _{k+1}, D_2]=\sum _{i=1}^{k-1}\alpha _i\partial _i+x_k^{2}\partial _k-2x_kx_{k+1}\partial _{k+1}.$$
But $2x_kx_{k+1}\partial _{k+1}\in s_n(K)$ and therefore $\sum _{i=1}^{k-1}\alpha _i\partial _i+x_k^{2}\partial _k\in S\setminus s_n(K).$ Besides, $\partial _k, x_k\partial _k\in s_n(K)$ and denoting
$$T_1=\partial _k, T_2=-\sum _{i=1}^{k-1}\alpha _i\partial _i-x_k^{2}\partial _k, T_3=-2x_k\partial _k$$
 we see by Lemma~\ref{sl2} that $S$ is  non-solvable. The latter contradicts  the choice of $S$ and this contradiction shows that the case $s>k$ is impossible.

	\underline{\it Case 2.}  Now let us consider the case $s=k.$  As shown above, in this case
	$\deg_{x_s}f_k=\deg_{x_k}f_k\geq 2.$
	Then the chosen derivation $D\in \mathcal D_k$ is of the form $$D=f_1\p_1+\dots+f_k\p_k,$$ where $\deg_{x_k}f_k\geq 2,$ $\deg_{x_i}f_k=0$ for all $i>k.$  Let us expand the polynomial $f_k\in P_{k}$ by powers of $x_k:$
	$$f_k(x_1, \dots, x_k)=g_0(x_1, \dots, x_{k-1})+g_1(x_1, \dots, x_{k-1})x_k+\dots+g_l(x_1, \dots, x_{k-1})x_k^l,$$ where $g_l\neq 0, \ l\geq 2, \ g_i\in P_{k-1}, i=1,\ldots ,l.$
	
	As in the previous case, let us consider the derivation
	\begin{multline*}D_0=[\underbrace{\p_k, \dots, [\p_k,}_{l-2~\text{times}} D]\dots ]=\alpha_1 \p_1+\dots +\alpha_{k-1}\p_{k-1}+\\+(u_0(x_1, \dots, x_{k-1})+u_1(x_1, \dots, x_{k-1})x_k+u_2(x_1, \dots, x_{k-1})x_k^2)\p_k \in S,
	\end{multline*}
	where $\alpha_i\in P_n, \ i=1, \ldots k-1,$ $u_0, u_1, u_2\in P_{k-1}$ and $u_2\neq 0.$
	Using a differential operator $\p_1^{\gamma_1}\dots\p_{k-1}^{\gamma_{k-1}}$ with appropriate $\gamma_i\geq 0, \ i=1,\ldots ,k-1,$ we can assume without loss of generality that $u_2(x_1, \dots, x_{k-1})=\lambda_k \in K^*$. We obtain the  derivation $$D_1=\mu_1 \p_1+\dots +\mu_{k-1}\p_{k-1}+(v_0(x_1, \dots, x_{k-1})+v_1(x_1, \dots, x_{k-1})x_k+\lambda_kx_k^2)\p_k \in S\setminus s_n(K),$$ where $\mu_i\in P_n, \ i=1,\ldots ,k-1,$ and $v_0, \ v_1\in P_{k-1}.$ Since $$(v_0(x_1, \dots, x_{k-1})+v_1(x_1, \dots, x_{k-1})x_k)\p_k\in s_n(K),$$ we have  the derivation $$D_2=\mu_1 \p_1+\dots +\mu_{k-1}\p_{k-1}+\lambda_kx_k^2\p_k\in S\setminus s_n(K).$$
	Denote by $L_1$ the subalgebra of $S$ generated by $T_1=\partial _k, T_2=-D_2, $ and $T_3=-2x_k\partial _k$. By Lemma~\ref{sl2} the subalgebra $L_1$ is non-solvable which contradicts to its choice. The obtained contradiction shows that  case $s=k$ is also impossible. Therefore our assumption about $S$ is false and $s_n(K)$ is a maximal solvable subalgebra of the Lie algebra $W_n(K).$
	\end{proof}


%
\end{document}